\documentclass[12pt]{amsart}

\usepackage[top=100pt,bottom=60pt,left=96pt,right=92pt]{geometry}

\usepackage[mathscr]{eucal}
\usepackage{hyperref}
\usepackage{amsmath,amsthm,amssymb,amsfonts,enumerate,xcolor,layout,tikz,subcaption}

\newtheorem{theorem}{Theorem}[section]

\theoremstyle{definition}

\theoremstyle{plain}

\definecolor{linkblue}{rgb}{0,0,.6}
\definecolor{citered}{rgb}{.7,0,0}
\hypersetup{colorlinks =true, linkcolor=linkblue, citecolor = citered}

\DeclareMathOperator{\E}{Exp}
\DeclareMathOperator{\Ht}{\widetilde{H}}
\DeclareMathOperator{\Tb}{\overline{\Theta}}
\DeclareMathOperator{\Tt}{\widetilde{\Theta}}
\DeclareMathOperator{\TT}{\mathcal{T}}
\DeclareMathOperator{\PP}{\mathcal{P}}

\begin{document}

\title[Theta Operators]{A proof of the Theta Operator Conjecture}
\author{Marino Romero}
\thanks{This work was supported by NSF DMS1902731}
\date{\today}

\begin{abstract}
In the context of the (generalized) Delta Conjecture and its compositional form, D'Adderio, Iraci, and Wyngaerd recently stated a conjecture relating two symmetric function operators, $D_k$ and $\Theta_k$.
We prove this Theta Operator Conjecture, finding it as a consequence of the five-term relation of Mellit and Garsia. 
 As a result, we find surprising ways of writing the $D_k$ operators.
\end{abstract}

\maketitle

\section{Introduction}

In what follows, we will assume the reader is familiar with symmetric functions and plethystic substitution. For a standard symmetric function reference, there is Macdonald's book \cite{Macdonald}. For some of the plethystic identities shown here, we will mostly reference  \cite{explicit} and \cite{positivity}. We will also adopt the notation and conventions in \cite{fiveterm}.
Garsia, Haiman, and Tesler defined in \cite{explicit} a family of plethystic operators $\{D_k\}_{k\in \mathbb{Z}}$ by setting
$$
D_k F[X] = F\left[X + \frac{M}{z} \right] \E[-zX] \Big|_{z^k}
$$
where 
$$\E[X] = \sum_{n \geq 0} h_n[X] = \exp\left(\sum_{k \geq 1} \frac{p_k}{k} \right)$$
is the plethystic exponential and $M=(1-q)(1-t)$. In the definition of $D_k$, we would have
$$
\E[-zX] = \sum_{k \geq 0 } (-z)^{k} e_k[X].
$$
For every partition $\mu$, set
$$
\Pi_\mu = \prod_{(i,j) \in \mu/(1)} (1-q^it^j)
$$
and define the linear operator $\Pi$ by setting 
$$
\Pi \Ht_\mu = \Pi_\mu \Ht_\mu.
$$
To get a compositional refinement of the (generalized) Delta Conjecture \cite{delta},  D'Adderio, Iraci, and Wyngaerd \cite{theta} define 
$$\Theta_k = \Pi~ \underline{e}_k^* ~\Pi^{-1}
$$
where $\underline{f}^*$ is multiplication by $f^*=f[X/M]$. 
\\

Conjecture 10.3 in \cite{theta} asserts that
$$
[\Theta_k, D_1] = \sum_{i=1}^k (-1)^{i} D_{i+1} \Theta_{k-i}.
$$
Multiplying both sides by $(-1)^k$ and expanding $[\Theta_k, D_1]$ as $\Theta_k D_1 - D_1 \Theta_k$ gives
$$
(-1)^k\Theta_k D_1= D_1 (-1)^k\Theta_k + \sum_{i=1}^k D_{i+1} (-1)^{k-i} \Theta_{k-i}.
$$
Let us set $\Tb_k = (-1)^k \Theta_k$, 
\begin{align*}
\Tb(z)  = \sum_{k \geq 0} z^k \Tb_k && \text{ and } && D_+(z) = \sum_{k \geq 1} z^k D_k.
\end{align*}
Then note that 
$$
D_+(z) \Tb(z) |_{z^{k+1}} = D_1 \Tb_k + D_2 \Tb_{k-1} + \cdots + D_{k+1} \Tb_0.
$$
The conjecture can then be rewritten as 
$$
z \Tb(z) D_1 = D_+(z) \Tb(z).
$$
This is what we will prove.
\section{The $D_k$ operators}
 Another way to write the $D_k$ operators is using the translation and multiplication operators $\TT_Y $ and $\PP_Z$, defined for any two expressions $Y$ and $Z$ by setting 
\begin{align*}
\TT_Y F[X] = F[X+Y] && \text{ and } && \PP_Z F[X] = \E[ZX]F[X].
\end{align*}
Then following the definition of $D_k$, we have as in \cite{explicit}
$$ 
\sum_{-\infty < k < \infty } z^k D_k = \PP_{-z} \TT_{\frac{M}{z}}.
$$

\section{The proof}
For any symmetric function $F$, define the linear operator $\Delta_F$ by setting \begin{align*}
\Delta_F \Ht_\mu = F[B_\mu] \Ht_\mu, && \text{ where} && B_\mu = \sum_{(i,j) \in \mu} q^i t^j. \end{align*}
This is the Delta eigenoperator for the modified Macdonald basis, defined in \cite{positivity}.
We will now follow the notation from Garsia and Mellit's five-term relation \cite{fiveterm}.
First note that if we set 
$$\Delta_u = \sum_{n \geq 0} (-u)^n\Delta_{e_n},$$ then we have 
\begin{align*}
\Delta_u \Ht_\mu = \Ht_\mu \prod_{(i,j)\in \mu} (1-uq^it^j) && \text{ and } && \Delta_{u}^{-1} = \sum_{ n \geq 0 } u^n \Delta_{h_n}.
\end{align*}
Furthermore, $\Pi = \Delta_u/(1-u) |_{u=1},$ and we can write
$$
\Tb_k = \Delta_u (-1)^k \underline{e}_k^* \Delta_u^{-1} \Big|_{u=1}.$$ 
Even though $\Delta_1^{-1}$ on its own is not well defined, one can still write in this case that $ \Tb(z) =\Delta_1 \PP_{-\frac{z}{M}} \Delta_1^{-1}$. We will not need this since the $u$'s will vanish in our calculations.
For this reason we will now, instead, consider the unspecialized operator
$$
\Tt_k = \Delta_v (-1)^k \underline{e}_k^* \Delta_v^{-1}
$$
and let 
$$
\Tt(z) = \sum_{k \geq 0}z^k \Tt_k = \Delta_v \PP_{-\frac{z}{M}} \Delta_{v}^{-1}
$$
for some monomial $v$.
We will show that
$$
z\Tt(z) D_1 =  D_+(z) \Tt(z),
$$
or rather
$$z\Tt(z) D_1 \Tt(z)^{-1} = D_+(z).
$$
In other words, we want to show that
$$
D_+(z) = z  \Delta_v \PP_{-\frac{z}{M}} \Delta_v^{-1} D_1  \Delta_v \PP_{\frac{z}{M}} \Delta_v^{-1}.
$$

Next, for $\mu \vdash n$, we set $\nabla \Ht_\mu = \Delta_u \Ht_\mu |_{u^n} = (-1)^n q^{n(\mu')} t^{n(\mu)} \Ht_\mu.$
With this convention, we have $D_1 = \nabla \underline{e}_1 \nabla^{-1}$ \cite{explicit}.  We can then rewrite the conjecture 
by substituting for $D_1$, giving
$$
D_+(z) = z \Delta_v \PP_{-\frac{z}{M}} \Delta_v^{-1} \nabla \underline{e}_1 \nabla^{-1}  \Delta_v \PP_{\frac{z}{M}} \Delta_v^{-1}.
$$

One of the main results from the five-term relation is the following identity:
\begin{theorem}[\cite{fiveterm}] For any two monomials $u$ and $v$, we have
$$
\nabla^{-1} \TT_{uv} \nabla = \Delta_v^{-1} \TT_u \Delta_{v} \TT_u^{-1}.
$$
\end{theorem}
The dual version is given by translating $\TT_z$ to $\PP_{-z/M}$ and reversing the order:
$$
\nabla \PP_{-\frac{uv}{M}} \nabla^{-1} =        \PP_{\frac{u}{M}}              \Delta_v \PP_{-\frac{u}{M} } \Delta_{v}^{-1}.
$$
We get
$$
  \Delta_v \PP_{-\frac{z}{M} } \Delta_{v}^{-1} =   \PP_{-\frac{z}{M}}\nabla \PP_{-\frac{zv}{M}} \nabla^{-1} 
$$
and the inverse formula
$$
   \Delta_v \PP_{\frac{z}{M} } \Delta_{v}^{-1} =  \nabla \PP_{\frac{zv}{M}} \nabla^{-1}  \PP_{\frac{z}{M}}.
$$

Substituting this in our conjectured formula, we get
\begin{align*}
\Delta_v \PP_{-\frac{z}{M}} \Delta_v^{-1} \nabla \underline{e}_1 \nabla^{-1}  \Delta_v \PP_{\frac{z}{M}} \Delta_v^{-1}  && =
&  \PP_{-\frac{z}{M}}\nabla \PP_{-\frac{zv}{M}} \nabla^{-1}  \nabla \underline{e}_1 \nabla^{-1}  \nabla \PP_{\frac{zv}{M}} \nabla^{-1} \PP_{\frac{z}{M}} \\
&& =  & 
  \PP_{-\frac{z}{M}}\nabla \PP_{-\frac{zv}{M}}  \underline{e}_1\PP_{\frac{zv}{M}} \nabla^{-1} \PP_{\frac{z}{M}} 
  \\
 && =  &  \PP_{-\frac{z}{M}}\nabla \underline{e}_1 \nabla^{-1} \PP_{\frac{z}{M}} 
   \\ &&=&    \PP_{-\frac{z}{M}}D_1 \PP_{\frac{z}{M}} .
\end{align*}
We can now write the final form of the conjecture. It is an extension of Proposition 1.6 in \cite{explicit}.
\begin{theorem}
$$
D_+(z) = \sum_{k \geq 1} z^k D_k =  z \PP_{-\frac{z}{M}} D_1 \PP_{\frac{z}{M}} 
$$
or 
$$
D_+(z) =  z\PP_{-\frac{z}{M}} \nabla \underline{e}_1 \nabla^{-1} \PP_{\frac{z}{M}} 
$$
\end{theorem}
\begin{proof} We will show that for any $r$
$$
\PP_{- \frac{z}{M}} D_r \PP_{\frac{z}{m}} = \sum_{k \geq 0} z^k D_{k+r}.
$$
In particular for $r=1$, this proves the theorem.
Following \cite{explicit}, for any two expressions $Y$ and $Z$, we have 
\begin{align*}
\TT_Y \PP_Z  ~ F[X] & = \TT_Y  ~ \E[XZ] F[X] \\
& =  \E[(X+Y)Z] F[X+Y] \\
& = \E[YZ] \E[XZ] F[X+Y] \\
& = \E[YZ] \PP_Z \TT_Y  ~ F[X],
\end{align*}
meaning
$$
\TT_Y \PP_Z = \E[YZ] \PP_Z \TT_Y.
$$
Therefore, we have
\begin{align*}
 \PP_{-\frac{z}{M}} \PP_{-u} \TT_{\frac{M}{u}} \PP_{\frac{z}{M}}  & =  \PP_{-\frac{z}{M}} \PP_{-u} \E\left[\frac{z}{u}\right] \PP_{\frac{z}{M}}\TT_{\frac{M}{u}} \\
& =   \frac{1}{1-z/u} \PP_{-u} \TT_{\frac{M}{u}}.
\end{align*}
The coefficient of $z^{k} u^r$ (with $k \geq 0$) on the right hand side  of the last equality is
$$
\PP_{-u}\TT_{\frac{M}{u}} \Big |_{u^{k+r}} = D_{k+r}.
$$
Equating coefficients on both sides, we have
$$
\PP_{-\frac{z}{M}} D_r  \PP_{\frac{z}{M}} \Big|_{z^{k}} = D_{k+r}.
$$

\end{proof}

\bibliographystyle{abbrv}
\bibliography{ThetaOperatorsBib}

\begin{thebibliography}{1}

\bibitem{fiveterm}
{A. M. Garsia, A. Mellit}.
\newblock {Five-term relation and Macdonald polynomials}.
\newblock {\em {J. Combin. Theory, Ser. A }}, 163:182--194, 2019.

\bibitem{plethystic}
{A. M. Garsia and G. Tesler}.
\newblock {Plethystic formulas for Macdonald $q, t$-Kostka coefficients}.
\newblock {\em {Adv. Math.}}, 123(2):144--222, Jan. 1996.

\bibitem{explicit}
{A. M. Garsia, M. Haiman, and G. Tesler}.
\newblock {Explicit Plethystic Formulas for Macdonald $q,t$-Kostka
  Coefficients}.
\newblock {\em The Andrews Festschrift. Seminaire Lotharingien}, 42:45 pp,
  1999.

\bibitem{positivity}
{F. Bergeron, A. M. Garsia, M. Haiman, G. Tesler.}
\newblock {Identities and Positivity Conjectures for some remarkable Operators
  in the Theory of Symmetric Functions}.
\newblock {\em {Methods Appl. Anal.}}, 6(3):363--420, 1999.

\bibitem{Macdonald}
{I. G. Macdonald}.
\newblock {\em {Symmetric Functions and Hall Polynomials}}.
\newblock Oxford University Press, second edition edition, 1995.

\bibitem{delta}
{J. Haglund, J. B. Remmel, and A. T. Wilson}.
\newblock {The Delta Conjecture}.
\newblock {\em { Trans. Amer. Math. Soc.}}, 370:4029--4057, Feb. 2018.

\bibitem{theta}
{M. D'Adderio , A. Iraci, and A. V. Wyngaerd}.
\newblock {Theta operators, refined Delta conjectures, and coinvariants}.
\newblock {\em arXiv:1906.02623}, 2019.

\end{thebibliography}

{\small
  \noindent
  \\
  Marino Romero\\
  University of Pennsylvania\\
  Department of Mathematics\\
  {\em E\--mail}: \texttt{mar007@sas.upenn.edu}
}

\end{document}